\documentclass[12pt,reqno]{article}

\usepackage[usenames]{color}
\usepackage{amssymb}
\usepackage{amsmath}
\usepackage{amsthm}
\usepackage{amsfonts}
\usepackage{amscd}
\usepackage{graphicx}

\usepackage[colorlinks=true,
linkcolor=webgreen,
filecolor=webbrown,
citecolor=webgreen]{hyperref}

\definecolor{webgreen}{rgb}{0,.5,0}
\definecolor{webbrown}{rgb}{.6,0,0}

\usepackage{color}
\usepackage{fullpage}
\usepackage{float}

\usepackage{graphics}
\usepackage{latexsym}

\begin{document}

\theoremstyle{plain}
\newtheorem{theorem}{Theorem}
\newtheorem{corollary}[theorem]{Corollary}
\newtheorem{lemma}{Lemma}
\newtheorem{example}{Examples}
\newtheorem*{remark}{Remark}

\begin{center}
\vskip 1cm
{\LARGE\bf
New harmonic number series
}

\vskip 1cm

{\large
Kunle Adegoke \\
Department of Physics and Engineering Physics, \\ Obafemi Awolowo University, 220005 Ile-Ife, Nigeria \\
\href{mailto:adegoke00@gmail.com}{\tt adegoke00@gmail.com}

\vskip 0.2 in

Robert Frontczak \\
Independent Researcher, 72762 Reutlingen \\ Germany \\
\href{mailto:robert.frontczak@web.de}{\tt robert.frontczak@web.de}
}

\end{center}

\vskip .2 in

\begin{abstract}
Based on a recent representation of the psi function due to Guillera and Sondow and independently Boyadzhiev,
new closed forms for various series involving harmonic numbers and inverse factorials are derived. 
A high point of the presentation is the rediscovery, by much simpler means, of a famous quadratic Euler sum 
originally discovered in 1995 by Borwein and Borwein.
\end{abstract}



\bigskip

\section{Introduction}

Our main purpose in this note is to discover the harmonic number series associated with the following identity:
\begin{equation}\label{main}
\sum_{n = 1}^\infty \frac{1}{n^2 \binom{{n + z}}{n}} = \zeta (2) - H_z^{(2)},\quad z\in\mathbb C\setminus Z^{-},
\end{equation}
where $\zeta(s)$ is the Riemann zeta function and $H_z^{(2)}$ is a second order harmonic number (both definitions are given below).
At $z=0$ identity~\eqref{main} subsumes the solution to the Basel problem and we will see that its derivative includes the well-known
relation between the Ap\'{e}ry constant and a classical Euler sum, namely,
\begin{equation*}
\sum_{n = 1}^\infty \frac{H_n}{n^2} = 2\,\zeta (3),
\end{equation*}
as a special case, $H_j$ being the $j$th harmonic number.

We will also evaluate the following series:
\begin{equation*}
\sum_{n = 1}^\infty \frac{1}{n (n + 1) \binom{{n + z}}{n}},\quad \sum_{n = 1}^\infty  {\frac{1}{{n\left( {n + 1} \right)\left( {n + 2} \right)\binom{{n + z}}{n}}}},\quad\sum_{n = 1}^\infty  {\frac{1}{{n\left( {n + 1} \right)\left( {n + 2} \right)\left( {n + 3} \right)\binom{{n + z}}{n}}}},
\end{equation*}
and derive the  harmonic and odd harmonic number series associated with them.

Identity~\eqref{main}, stated without proof in Sofo and Srivastava~\cite[Equation (2.13)]{sofo15}, is a consequence of the following representation of the digamma function $\psi(z)=\Gamma'(z)/\Gamma(z)$:
\begin{equation}\label{eq.v94kn2f}
\psi (z) = \sum_{n = 0}^\infty \frac{1}{{n + 1}}\sum_{k = 0}^n \binom{{n}}{k} (- 1)^k \ln \left( {z + k} \right),
\end{equation}
which holds for all $z\in\mathbb{C}$ with $\Re (z)>0$. This representation first appeared in~\cite[Theorem 5.1]{guillera08}
and was rediscovered by Boyadzhiev~\cite[Identity (5.18)]{Boyadzhiev1}.

Harmonic numbers $H_\alpha$ and odd harmonic numbers $O_\alpha$ are defined
for $0\ne\alpha\in\mathbb C\setminus\mathbb Z^{-}$ by the recurrence relations
\begin{equation*}
H_\alpha = H_{\alpha - 1} + \frac{1}{\alpha} \qquad \text{and} \qquad O_\alpha = O_{\alpha - 1} + \frac{1}{2\alpha - 1},
\end{equation*}
with $H_0=0$ and $O_0=0$. Harmonic numbers are connected to the digamma function through the fundamental relation
\begin{equation}\label{Harm_psi}
H_\alpha = \psi(\alpha + 1) + \gamma.
\end{equation}

Generalized harmonic numbers $H_\alpha^{(m)}$ and odd harmonic numbers $O_\alpha^{(m)}$ of order $m\in\mathbb C$ are defined by
\begin{equation*}
H_\alpha^{(m)} = H_{\alpha - 1}^{(m)} + \frac{1}{\alpha^m} \qquad \text{and} \qquad
O_\alpha^{(m)} = O_{\alpha - 1}^{(m)} + \frac{1}{(2\alpha - 1)^m},
\end{equation*}
with $H_0^{(m)}=0$ and $O_0^{(m)}=0$ so that $H_\alpha=H_\alpha^{(1)}$ and $O_\alpha=O_\alpha^{(1)}$.

The recurrence relations imply that if $\alpha=n$ is a non-negative integer, then
\begin{equation*}
H_n^{(m)} = \sum_{j = 1}^n \frac{1}{j^m} \qquad \text{and} \qquad O_n^{(m)} = \sum_{j = 1}^n \frac{1}{(2j - 1)^m}.
\end{equation*}

Generalized harmonic numbers are linked to the polygamma functions $\psi^{(r)}(z)$ of order $r$ defined by
\begin{equation*}
\psi^{(r)} (z) = \frac{d^r}{dz^r}\psi(z) = (- 1)^{r + 1} r!\sum_{j = 0}^\infty \frac{1}{{( j + z )^{r + 1} }},
\end{equation*}
through
\begin{equation*}
H_z^{(r)} = \zeta (r) + \frac{(- 1)^{r - 1}}{(r - 1)!} \psi^{(r - 1)} (z + 1),
\end{equation*}
where $\zeta(s)$ is the Riemann zeta function defined by
\begin{equation*}
\zeta(s) = \sum_{k=1}^\infty \frac{1}{k^s}, \quad s\in\mathbb{C}, \Re(s)>1.
\end{equation*}
The analytical continuation to all $s\in\mathbb{C}$ with $\Re(s)>0,s\neq 1,$ is given by
\begin{equation*}
\zeta(s) = (1-2^{1-s})^{-1} \sum_{k=1}^\infty \frac{(-1)^{k+1}}{k^s}.
\end{equation*}

\section{Proof of identity~\eqref{main}}

For completeness and a better readability we first prove identity~\eqref{main}.
\begin{theorem}
For all $z\in\mathbb C\setminus Z^{-}$ the following identity holds:
\begin{equation*}
\sum_{n = 1}^\infty \frac{1}{n^2 \binom{{n + z}}{n}} = \zeta (2) - H_z^{(2)}.
\end{equation*}
\end{theorem}
\begin{proof}
Differentiating the representation~\eqref{eq.v94kn2f} gives
\begin{equation*}
\begin{split}
\psi^{(1)}(z) = \zeta (2) - H_{z - 1}^{(2)} &= \sum_{n = 0}^\infty \frac{1}{n + 1}\sum_{k = 0}^n \binom{n}{k}\frac{(- 1)^k}{z + k} \\
& = \sum_{n = 1}^\infty \frac{1}{n} \sum_{k = 1}^n \binom{{n - 1}}{{k - 1}}\frac{(- 1)^{k - 1}}{z + k - 1} \\
& = \sum_{n = 1}^\infty \frac{1}{n} \sum_{k = 1}^n \frac{k}{n}\binom{{n}}{k}\frac{(- 1)^{k - 1}}{z + k - 1} \\
& = \sum_{n = 1}^\infty \frac{1}{n^2} \sum_{k = 1}^n k \binom{{n}}{k}\frac{(- 1)^{k - 1}}{z + k - 1} \\
& = \sum_{n = 1}^\infty \frac{1}{n^2} \frac{1}{\binom{{n + z - 1}}{n}};
\end{split}
\end{equation*}
and hence~\eqref{main}. Note that we used
\begin{equation*}
\sum_{k = 1}^n \binom{{n}}{k} \frac{(- 1)^{k - 1} k}{z + k} = \frac{1}{\binom{{n + z}}{n}},
\end{equation*}
which is a particular case of the Frisch identity (\cite{Abel,Adegoke}).
\end{proof}

\section{Required identities}

We will make frequent use of the following basic identity~\cite{rocket81}:
\begin{equation}\label{bs}
\sum_{k = 1}^m \frac{1}{\binom{{k + n}}{k}} = \frac{1}{{n - 1}} - \frac{n}{n - 1}\frac{1}{\binom{{m + n}}{{m + 1}}},\quad 0,1\ne n\in\mathbb C.
\end{equation}

and the identities stated in the following lemmata.

\begin{lemma}\label{lem.ho}
We have
\begin{gather}
H_{k - 1/2} = 2\,O_k - 2\,\ln 2, \\
H_{k - 1/2}^{(2)} = - 2\,\zeta \left( 2 \right) + 4\,O_k^{(2)}, \\
H_{- 1/2}^{(3)} = - 6\,\zeta \left( 3 \right), \\
H_{- 1/2}^{(4)} = - 14\,\zeta \left( 4 \right), \\
H_{k - 1/2}^{(m + 1)} - H_{- 1/2}^{(m + 1)} = 2^{m + 1} O_k^{(m + 1)}.
\end{gather}
\end{lemma}

\begin{lemma}\label{lem.bin}
For integers $u$ and $v$, we have
\begin{gather}
\binom{{u - 1/2}}{v} = \binom{{2u}}{u}\binom{{u}}{v}2^{ - 2v} \binom{{2(u - v)}}{{u - v}}^{-1},\\
\binom{{u}}{{1/2}} = \frac{{2^{2u + 1} }}{\pi }\binom{{2u}}{u}^{ - 1},\\
\binom{{u}}{{1/2 - v}} = \frac{{( - 1)^{v - 1} }}{v}\binom{{u + v}}{v}^{ - 1} \binom{{2(u + v)}}{{u + v}}^{ - 1} \binom{{2(v - 1)}}{{v - 1}}2^{2u + 2},\\
\binom{{u + 1/2}}{v} = \binom{{u}}{v}^{ - 1} \binom{{2u + 1}}{{2v}}2^{ - 2v} \binom{{2v}}{v},\\
\binom{{u + 1/2}}{v} = \frac{{( - 1)^{v - u - 1} }}{{2^{2v - 1} }}\,\frac{{2u + 1}}{{v - u}}\binom{{v}}{u}^{ - 1} \binom{{2u}}{u}\binom{{2(v - u - 1)}}{{v - u - 1}},\quad v>u,\\
\binom{{ - 3/2}}{u} = ( - 1)^u \,(2u + 1)\,2^{-2u}\binom{{2u}}{u}.
\end{gather}
\end{lemma}

\begin{proof}
These identities are readily derived using the following well-known Gamma function identities:
\begin{equation*}
\Gamma \left( {u + \frac{1}{2}} \right) = \frac{{\sqrt \pi  }}{{2^{2u} }}\binom{{2u}}{u}\Gamma \left( {u + 1} \right),\quad\Gamma \left( { - u + \frac{1}{2}} \right) = ( - 1)^u 2^{2u} \binom{{2u}}{u}^{ - 1} \frac{{\sqrt \pi  }}{{\Gamma \left( {u + 1} \right)}},
\end{equation*}
together with the definition of the generalized binomial coefficients:
\begin{equation*}
\binom{{u}}{v} = \frac{{\Gamma \left( {u + 1} \right)}}{{\Gamma \left( {v + 1} \right)\Gamma \left( {u - v + 1} \right)}}.
\end{equation*}
\end{proof}

\section{Results}

In this section we state new closed forms for infinite series involving inverse factorials. More results of this nature
have been produced, among others, by Sofo \cite{Sofo1,Sofo2}, Boyadzhiev~\cite{Boyadzhiev2,Boyadzhiev3} and by the authors \cite{Adegoke2}.

\begin{theorem}\label{thm.f6brbmb}
If $m$ is a non-negative integer, then
\begin{equation}\label{eq.nw1yuwc}
\sum_{n = 1}^\infty \frac{{2^{2n}}}{{n^2 \binom{{2\left( {n + m} \right)}}{{n + m}}\binom{{n + m}}{m}}}
= \frac{1}{\binom{{2m}}{m}}\left( 3\,\zeta (2) - 4\,O_m^{(2)} \right).
\end{equation}
In particular,
\begin{equation}
\sum_{n = 1}^\infty \frac{{2^{2n}}}{n^2 \binom{{2n}}{n}} = \frac{\pi^2}{2}.
\end{equation}
\end{theorem}
\begin{proof}
Write $m - 1/2$ for $z$ in~\eqref{main} to obtain
\begin{equation}
\sum_{n = 1}^\infty \frac{1}{n^2 \binom{{n + m - 1/2}}{n}} = \zeta (2) - H_{m - 1/2}^{(2)} ,
\end{equation}
from which~\eqref{eq.nw1yuwc} follows on account of Lemmata~\ref{lem.ho} and~\ref{lem.bin}.
\end{proof}

\begin{theorem}\label{thm.zj0cogg}
If $z\in\mathbb C\setminus\mathbb Z^{-}$, then
\begin{equation}\label{eq.jdzztai}
\sum_{n = 1}^\infty \frac{H_{n + z}}{n^2 \binom{{n + z}}{n}} = H_z \left( \zeta (2) - H_z^{(2)} \right) + 2\left( \zeta (3) - H_z^{(3)} \right).
\end{equation}
In particular,
\begin{gather}
\sum_{n = 1}^\infty  {\frac{{H_n }}{{n^2 }}}  = 2\,\zeta (3),\\
\sum_{n = 1}^\infty  {\frac{{H_n }}{{n^2 \left( {n + 1} \right)}}}  = 2\,\zeta (3) - \zeta (2)\label{eq.b4jk2f8} ,\\
\sum_{n = 1}^\infty  {\frac{{H_n }}{{n^2 \left( {n + 1} \right)\left( {n + 2} \right)}}}  = \zeta (3) - \frac{3}{4}\,\zeta (2) + \frac{1}{4}.
\end{gather}
\end{theorem}
\begin{proof}
Differentiate~\eqref{main} with respect to $z$ to obtain
\begin{equation}\label{eq.lchpe3r}
\sum_{n = 1}^\infty \frac{H_{n + z} - H_z}{n^2 \binom{{n + z}}{n}} = 2\left( \zeta (3) - H_z^{(3)} \right),
\end{equation}
and use~\eqref{main} again to rewrite the second term on the left hand side of~\eqref{eq.lchpe3r}. 

Identity~\eqref{eq.b4jk2f8} corresponds to an evaluation of~\eqref{eq.jdzztai} at $z=1$ followed by the use of
\begin{equation}\label{eq.f9xz0ur}
H_{n+1}=H_n + \frac1{n+1},
\end{equation}
and the decomposition
\begin{equation*}
\frac{1}{{n^2 \left( {n + 1} \right)^2 }} = \frac{1}{{n^2 }} + \frac{1}{{\left( {n + 1} \right)^2 }} - 2\left( {\frac{1}{n} - \frac{1}{{n + 1}}} \right).
\end{equation*}
\end{proof}

\begin{corollary}
If $m$ is a nonnegative integer, then
\begin{equation}
\sum_{n = 1}^\infty \frac{2^{2n}\, O_{n + m}}{{n^2\, \binom{{2\left( {n + m} \right)}}{{n + m}}\binom{{n + m}}{m}}}
= \frac{1}{\binom{{2m}}{m}}\left( O_m \left( 3\,\zeta (2) - 4\,O_m^{(2)} \right) + \left( 7\,\zeta (3) - 8\,O_m^{(3)} \right) \right).
\end{equation}
\end{corollary}
\begin{proof}
Write $m - 1/2$ for $z$ in~\eqref{eq.lchpe3r} and use Lemmata~\ref{lem.ho} and~\ref{lem.bin}.
\end{proof}

\begin{theorem}
If $z\in\mathbb C\setminus Z^{-}$, then
\begin{equation}
\sum_{n = 1}^\infty \frac{\left( H_{n + z} - H_z \right)^2 + H_{n + z}^{(2)} - H_z^{(2)}}{n^2 \binom{{n + z}}{z}} = 6\,\zeta (4) - 6\,H_z^{(4)}.
\end{equation}
In particular,
\begin{gather}
\sum_{n = 1}^\infty  {\frac{{H_n^2  + H_n^{(2)} }}{{n^2 }}}  = 6\,\zeta (4),\label{eq.hyyfilz}\\
\sum_{n = 1}^\infty  {\frac{1}{{n^2 }}\frac{{2^{2n} }}{{\binom{{2n}}{n}}}\left( {O_n^2  + O_n^{(2)} } \right)}  = \frac{{\pi ^2 }}{4}.
\end{gather}
\end{theorem}
\begin{proof}
Differentiate~\eqref{eq.lchpe3r} with respect to $z$.
\end{proof}

\begin{remark}
The symmetry relation
\begin{equation*}
\sum_{n = 1}^\infty \frac{{H_n^{(p)}}}{{n^q}} + \sum_{n = 1}^\infty \frac{{H_n^{(q)}}}{{n^p}} = \zeta (p)\zeta (q) + \zeta (p + q)
\end{equation*}
makes it easy to calculate
\begin{equation*}
\sum_{n = 1}^\infty \frac{{H_n^{(p)}}}{{n^p}} = \frac{1}{2}\left( \zeta (p)^2 + \zeta (2p) \right),
\end{equation*}
so that, in particular,
\begin{equation*}
\sum_{n = 1}^\infty \frac{{H_n^{(2)}}}{{n^2}} = \frac{7}{4}\,\zeta (4);
\end{equation*}
and using this in~\eqref{eq.hyyfilz} therefore gives
\begin{equation}\label{bowen}
\sum_{n = 1}^\infty \frac{{H_n^2}}{{n^2}} = \frac{17}{4}\,\zeta (4).
\end{equation}
Thus~\eqref{eq.hyyfilz} allows a much easier determination of the sum~\eqref{bowen} which was originally evaluated by
Borwein and Borwein~\cite{bowen95} through Fourier series and contour integration.
\end{remark}

\begin{lemma}
If $z\in\mathbb C\setminus\mathbb Z^{-}$, then
\begin{equation}\label{eq.bhasxe8}
\sum_{n = 1}^\infty \frac{1}{n (n + 1) \binom{{n + z}}{n}} = z\left( {H_z^{(2)} - \zeta (2)}\right) + 1,
\end{equation}
\end{lemma}
\begin{proof}
Sum~\eqref{main} over $z$ from $1$ to $r$, using~\eqref{bs} and the fact that~\cite[Equation (3.1)]{adegoke16}
\begin{equation*}
\sum_{z = 1}^r {H_z^{(2)} }  = \left( {r + 1} \right)H_r^{(2)}  - H_r,
\end{equation*}
to obtain
\begin{equation*}
\sum_{n = 1}^\infty \frac{1}{n (n + 1)\binom{{n + r}}{r}} = \sum_{n = 1}^\infty \frac{1}{(n + 1)^2 n} - 1 - (r - 1)\zeta (2) + r H_r^{(2)}.
\end{equation*}
But
\begin{equation*}
\begin{split}
\sum_{n = 1}^\infty \frac{1}{(n + 1)^2 n} &= \sum_{n = 1}^\infty \left( \frac{1}{n} - \frac{1}{{n + 1}} \right)
- \sum_{n = 1}^\infty  \frac{1}{(n + 1)^2} \\
& = \sum_{n = 1}^\infty \left( \frac{1}{n} - \frac{1}{{n + 1}} \right) - \sum_{n = 2}^\infty \frac{1}{n^2} \\
& = \sum_{n = 1}^\infty \left( \frac{1}{n} - \frac{1}{{n + 1}} \right) + 1 - \sum_{n = 1}^\infty \frac{1}{n^2} \\
& = 2 - \zeta (2),
\end{split}
\end{equation*}
which when substituted into the previous sum gives~\eqref{eq.bhasxe8} after replacing $r$ with $z$.
\end{proof}

\begin{remark}
Identity~\eqref{eq.bhasxe8} was also derived by~Sofo and Srivastava~\cite[Equation (2.17)]{sofo15}.
\end{remark}

\begin{theorem}
If $m$ is a non-negative integer, then
\begin{equation}
\sum_{n = 1}^\infty \frac{{2^{2n} }}{{n\left( {n + 1} \right)\binom{{2\left( {n + m} \right)}}{{n + m}}\binom{{n + m}}{m}}} 
= \frac{1}{{\binom{{2m}}{m}}}\left( {\left( {m - \frac{1}{2}} \right)\left( {4\,O_m^{(2)}  - 3\,\zeta (2)} \right) + 1} \right).
\end{equation}
In particular,
\begin{gather}
\sum_{n = 1}^\infty \frac{{2^{2n} }}{{n\left( {n + 1} \right)\binom{{2n}}{n}}} = \frac{{\pi ^2 }}{4} + 1,\\
\sum_{n = 1}^\infty \frac{{2^{2n} }}{{n\left( {n + 1} \right)^2 \binom{{2\left( {n + 1} \right)}}{{n + 1}}}} = \frac{3}{2} - \frac{{\pi ^2 }}{8},\\
\sum_{n = 1}^\infty \frac{{2^{2n} }}{{n\left( {n + 1} \right)^2 \left( {n + 2} \right)\binom{{2\left( {n + 2} \right)}}{{n + 2}}}} 
= \frac{{23}}{{36}} - \frac{{\pi ^2 }}{{16}}.
\end{gather}
\end{theorem}

\begin{theorem}
If $z\in\mathbb C\setminus\mathbb Z^{-}$, then
\begin{equation}\label{eq.pwogiuk}
\sum_{n = 1}^\infty  {\frac{{H_{n + z} }}{{n\left( {n + 1} \right)\binom{{n + z}}{n}}}}  = \left( {\zeta (2) - H_z^{(2)} } \right)\left( {1 - zH_z } \right) + H_z  - 2z\left( {\zeta (3) - H_z^{(3)} } \right).
\end{equation}
In particular,
\begin{gather}
\sum_{n = 1}^\infty  {\frac{{H_n }}{{n\left( {n + 1} \right)}}}  = \frac{{\pi ^2 }}{6},\label{eq.tgvrkzb}\\
\sum_{n = 1}^\infty  {\frac{{H_n }}{{n\left( {n + 1} \right)^2 }}}  = \zeta (2) - \zeta (3)\label{eq.v8qrbaf}.
\end{gather}
\end{theorem}
\begin{proof}
Differentiate~\eqref{eq.bhasxe8} with respect to $z$ to obtain
\begin{equation}\label{eq.ta7cqa4}
\sum_{n = 1}^\infty \frac{H_{n + z} - H_z}{{n (n + 1) \binom{{n + z}}{n}}} = \zeta (2) - H_z^{(2)} - 2z\left( \zeta (3) - H_z^{(3)} \right),
\end{equation}
and hence~\eqref{eq.pwogiuk} upon using~\eqref{eq.bhasxe8} again to rewrite the second sum on the left hand side of~\eqref{eq.ta7cqa4}. Identity \eqref{eq.v8qrbaf} is an evaluation of~\eqref{eq.pwogiuk} at $z=1$ where we used~\eqref{eq.f9xz0ur} and the fact that
\begin{equation*}
\begin{split}
\sum_{n = 1}^\infty \frac{1}{n (n + 1)^3} &= \sum_{n = 1}^\infty \left( \frac{1}{n} - \frac{1}{{n + 1}} - \frac{1}{(n + 1)^2} 
- \frac{1}{(n + 1)^3} \right) \\
&= \sum_{n = 1}^\infty \left( \frac{1}{n} - \frac{1}{{n + 1}} \right) - \sum_{n = 1}^\infty \frac{1}{(n + 1)^2} - \sum_{n = 1}^\infty \frac{1}{(n + 1)^3}. 
\end{split}
\end{equation*}
\end{proof}

\begin{remark}
Identity~\eqref{eq.tgvrkzb} was also recorded by Chu in~\cite{chu12}.
\end{remark}

\begin{remark}
Subtraction of~\eqref{eq.v8qrbaf} from~\eqref{eq.b4jk2f8} gives the Euler sum
\begin{equation}
\sum_{n = 1}^\infty \frac{{H_n }}{{n^2 \left( {n + 1} \right)^2 }} = 3\,\zeta (3) - 2\,\zeta (2).
\end{equation}
\end{remark}

\begin{corollary}
If $m$ is a non-negative integer, then
\begin{equation}
\begin{split}
\sum_{n = 1}^\infty  {\frac{{2^{2n} O_{n + m} }}{{n\left( {n + 1} \right)\binom{{2\left( {n + m} \right)}}{{n + m}}\binom{{n + m}}{m}}}}  &= \frac{1}{{2\binom{{2m}}{m}}}\left( {3\zeta (2) - 4O_m^{(2)} } \right)\left( {1 - O_m \left( {2m - 1} \right)} \right) \\
&\qquad + \frac{O_m}{{\binom{{2m}}{m}}} - \frac{{\left( {2m - 1} \right)}}{{2\binom{{2m}}{m}}}\left( {7\zeta (3) - 8O_m^{(3)} } \right).
\end{split}
\end{equation}
In particular,
\begin{gather}
\sum_{n = 1}^\infty  {\frac{{2^{2n} O_n }}{{n\left( {n + 1} \right)\binom{{2n}}{n}}}}  = \frac{{\pi ^2 }}{4} + \frac{7}{2}\zeta (3),\\
\sum_{n = 1}^\infty  {\frac{{2^{2n} O_{n + 1} }}{{n\left( {n + 1} \right)^2 \binom{{2\left( {n + 1} \right)}}{{n + 1}}}}}  = \frac{5}{2} - \frac{7}{4}\zeta (3).
\end{gather}
\end{corollary}
\begin{proof}
Write $m-1/2$ for $z$ in~\eqref{eq.ta7cqa4} and use Lemmata~\ref{lem.ho} and~\ref{lem.bin}.
\end{proof}

\begin{theorem}
If $z\in\mathbb C\setminus\mathbb Z^{-}$, then
\begin{equation}
\sum_{n = 1}^\infty \frac{\left( H_{n + z} - H_z \right)^2 + H_{n + z}^{(2)} - H_z^{(2)}}{n (n + 1)\binom{{n + z}}{z}} 
= 4\left( {\zeta (3) - H_z^{(3)} } \right) - 6z\left( {\zeta (4) - H_z^{(4)} } \right).
\end{equation}
In particular,
\begin{gather}
\sum_{n = 1}^\infty \frac{H_n^2 + H_n^{(2)}}{n (n + 1)} = 4\,\zeta (3),\label{eq.v82402j} \\
\sum_{n = 1}^\infty \frac{2^{2n} \left( O_n^2 + O_n^{(2)} \right)}{n (n + 1) \binom{{2n}}{n}} = \frac{\pi^4}{8} + 7\,\zeta (3).
\end{gather}
\end{theorem}
\begin{proof}
Differentiate~\eqref{eq.ta7cqa4} with respect to $z$.
\end{proof}

\begin{remark}
Since Xu showed that ~\cite[Equation (2.30)]{xu17}
\begin{equation*}
\sum_{n = 1}^\infty \frac{H_n^{(2)}}{n\left( {n + 1} \right)} = \zeta (3),
\end{equation*}
identity~\eqref{eq.v82402j} yields
\begin{equation}
\sum_{n = 1}^\infty \frac{{H_n^2 }}{{n\left( {n + 1} \right)}} = 3\,\zeta (3);
\end{equation}
an identity that was also reported by Nimbran and Sofo in~\cite[Identity (2.24)]{nimbran19}.
\end{remark}

\begin{lemma}
If $z\in\mathbb C\setminus\mathbb Z^{-}$, then,
\begin{equation}\label{eq.lmcv0zf}
\sum_{n = 1}^\infty \frac{1}{{n\left( {n + 1} \right)\left( {n + 2} \right)\binom{{n + z}}{n}}}
=\frac{1}{2}z\left( {z + 1} \right)\left( {H_z^{(2)}  - \zeta (2)} \right) + \frac{z}{2} + \frac{1}{4}.
\end{equation}
\end{lemma}
\begin{proof}
We first derive the following identity:
\begin{equation}\label{eq.dhruknq}
\sum_{n = 1}^\infty  {\frac{1}{{n\left( {n + 2} \right)\binom{{n + z}}{n}}}}=\frac{1}{2}z\left( {z - 1} \right)\left( {\zeta (2) - H_z^{(2)} } \right) - \frac{z}{2} + \frac{3}{4},
\end{equation}
by summing both sides of~\eqref{eq.bhasxe8} over $z$ from 1 to $r$ and replacing $r$ with $z$ in the final identity. Note that
\begin{equation*}
\sum_{n = 1}^\infty  {\frac{1}{{n\left( {n + 1} \right)\left( {n + 2} \right)}}}  = \frac{1}{4},
\end{equation*}
since
\begin{equation*}
\frac{1}{{n\left( {n + 1} \right)\left( {n + 2} \right)}} = \frac{1}{2}\left( {\frac{1}{n} - \frac{1}{{n + 1}} - \frac{1}{{n + 1}} + \frac{1}{{n + 2}}} \right).
\end{equation*}
Note also that~\cite[Equation (3.2)]{adegoke16}:
\begin{equation}\label{eq.izgpi03}
2\sum_{z = 1}^r {zH_z^{(2)} }  = r\left( {r + 1} \right)H_r^{(2)}  + H_r  - r.
\end{equation}
Identity~\eqref{eq.lmcv0zf} is obtained by subtracting~\eqref{eq.dhruknq} from~\eqref{eq.bhasxe8}.
\end{proof}

\begin{theorem}
If $m$ is a non-negative integer, then
\begin{equation}
\begin{split}
&\sum_{n = 1}^\infty  {\frac{{2^{2n} }}{{n\left( {n + 1} \right)\left( {n + 2} \right)\binom{{2\left( {n + m} \right)}}{{n + m}}\binom{{n + m}}{m}}}}  \\
&\qquad= \frac{1}{{2\binom{{2m}}{m}}}\left(\left( {m - \frac{1}{2}} \right)\left( {m + \frac{1}{2}} \right)\left( {4O_m^{(2)}  - 3\zeta (2)} \right) + m\right).
\end{split}
\end{equation}
In particular,
\begin{gather}
\sum_{n = 1}^\infty  {\frac{{2^{2n} }}{{n\left( {n + 1} \right)\left( {n + 2} \right)\binom{{2n}}{n}}}}  = \frac{{\pi ^2 }}{{16}},\\
\sum_{n = 1}^\infty  {\frac{{2^{2n} }}{{n\left( {n + 1} \right)^2 \left( {n + 2} \right)\binom{{2\left( {n + 1} \right)}}{{n + 1}}}}}  = 1 - \frac{{3\pi ^2 }}{{32}},\\
\sum_{n = 1}^\infty  {\frac{{2^{2n + 1} }}{{n\left( {n + 1} \right)^2 \left( {n + 2} \right)^2 \binom{{2\left( {n + 2} \right)}}{{n + 2}}}}}  = \frac{{14}}{9} - \frac{{5\pi ^2 }}{{32}}.
\end{gather}
\end{theorem}
\begin{proof}
Set $z=m - 1/2$ in~\eqref{eq.lmcv0zf}.
\end{proof}

\begin{theorem}
If $z\in\mathbb C\setminus\mathbb Z^{-}$, then
\begin{equation}\label{eq.l2ciolu}
\begin{split}
\sum_{n = 1}^\infty  {\frac{{H_{n + z} }}{{n\left( {n + 1} \right)\left( {n + 2} \right)\binom{{n + z}}{n}}}}  &= \left( {z + \frac{1}{2} - \frac{1}{2}z\left( {z + 1} \right)}H_z \right)\left( {\zeta (2) - H_z^{(2)} } \right) \\
&\qquad + \left( {\frac{z}{2} + \frac{1}{4}} \right)H_z- z\left( {z + 1} \right)\left( {\zeta (3) - H_z^{(3)} } \right) - \frac{1}{2}.
\end{split}
\end{equation}
In particular,
\begin{gather}
\sum_{n = 1}^\infty  {\frac{{H_n }}{{n\left( {n + 1} \right)\left( {n + 2} \right)}}}  = \frac{{\pi ^2 }}{{12}} - \frac{1}{2},\\
\sum_{n = 1}^\infty  {\frac{{H_n }}{{n\left( {n + 1} \right)^2 \left( {n + 2} \right)}}}  = \frac{{\pi ^2 }}{{12}} + \frac{1}{2} - \zeta (3)\label{eq.r7il44k}.
\end{gather}
\end{theorem}
\begin{proof}
Differentiate~\eqref{eq.lmcv0zf} to obtain
\begin{equation}\label{eq.y8vfuxa}
\begin{split}
\sum_{n = 1}^\infty  {\frac{{H_{n + z}  - H_z }}{{n\left( {n + 1} \right)\left( {n + 2} \right)\binom{{n + z}}{n}}}}  &= \left( {z + \frac{1}{2}} \right)\left( {\zeta (2) - H_z^{(2)} } \right)\\
&\qquad - z\left( {z + 1} \right)\left( {\zeta (3) - H_z^{(3)} } \right) - \frac{1}{2}.
\end{split}
\end{equation}
Identity~\eqref{eq.r7il44k} is an evaluation of~\eqref{eq.l2ciolu} at $z=1$, where we also used~\eqref{eq.f9xz0ur} and the fact that
\begin{equation*}
\sum_{n = 1}^\infty  {\frac{1}{{n\left( {n + 1} \right)^3 \left( {n + 2} \right)}}}  = \frac{5}{4} - \zeta (3),
\end{equation*}
since
\begin{equation*}
\frac{1}{{n\left( {n + 1} \right)^3 \left( {n + 2} \right)}} = \frac{1}{2}\left( {\frac{1}{n} - \frac{1}{{n + 1}}} \right) - \frac{1}{2}\left( {\frac{1}{{n + 1}} - \frac{1}{{n + 2}}} \right) - \frac{1}{{\left( {n + 1} \right)^3 }}.
\end{equation*}
\end{proof}

\begin{theorem}
If $m$ is a non-negative integer, then
\begin{equation}
\begin{split}
&\sum_{n = 1}^\infty  {\frac{{2^{2n} O_{n + m} }}{{n\left( {n + 1} \right)\left( {n + 2} \right)\binom{{2\left( {n + m} \right)}}{{n + m}}\binom{{n + m}}{m}}}}\\
&\qquad = \frac{1}{{2\binom{{2m}}{m}}}\left( {m - O_m \left( {m - \frac{1}{2}} \right)\left( {m + \frac{1}{2}} \right)} \right)\left( {3\zeta (2) - 4O_m^{(2)} } \right)\\
&\qquad\quad - \frac{1}{{2\binom{{2m}}{m}}}\left( {\left( {m - \frac{1}{2}} \right)\left( {m + \frac{1}{2}} \right)\left( {7\zeta (3) - 8O_m^{(3)} } \right) - mO_m  + \frac{1}{2}} \right).
\end{split}
\end{equation}
In particular,
\begin{gather}
\sum_{n = 1}^\infty  {\frac{{2^{2n} O_n }}{{n\left( {n + 1} \right)\left( {n + 2} \right)\binom{{2n}}{n}}}}  = \frac{7}{8}\,\zeta (3) - \frac{1}{4},\\
\sum_{n = 1}^\infty  {\frac{{2^{2n} O_{n + 1} }}{{n\left( {n + 1} \right)^2 \left( {n + 2} \right)\binom{{2\left( {n + 1} \right)}}{{n + 1}}}}}  = \frac{{\pi ^2 }}{{32}} - \frac{{21}}{{16}}\,\zeta (3) + \frac{{11}}{8}.
\end{gather}
\end{theorem}
\begin{proof}
Set $z=m-1/2$ in~\eqref{eq.y8vfuxa} and use Lemmata~\ref{lem.ho} and~\ref{lem.bin}.
\end{proof}

\begin{theorem}
If $z\in\mathbb C\setminus\mathbb Z^{-}$, then
\begin{equation}
\begin{split}
\sum_{n = 1}^\infty  {\frac{{\left( {H_{n + z}  - H_z } \right)^2  + H_{n + z}^{(2)}  - H_z^{(2)} }}{{n\left( {n + 1} \right)\left( {n + 2} \right)\binom{{n + z}}{z}}}}  &= H_z^{(2)}  - \zeta (2) - 2\left( {2z + 1} \right)\left( {H_z^{(3)}  - \zeta (3)} \right)\\
&\qquad+ 3z\left( {z + 1} \right)\left( {H_z^{(4)}  - \zeta (4)} \right).
\end{split}
\end{equation}
In particular,
\begin{gather}
\sum_{n = 1}^\infty  {\frac{{H_n^2  + H_n^{(2)} }}{{n\left( {n + 1} \right)\left( {n + 2} \right)}}}  = 2\zeta (3) - \zeta (2),\\
\sum_{n = 1}^\infty  {\frac{{2^{2n} \left( {O_n^2  + O_n^{(2)} } \right)}}{{n\left( {n + 1} \right)\left( {n + 2} \right)\binom{{2n}}{n}}}}  = \frac{{\pi ^4 }}{{32}} - \frac{{\pi ^2 }}{8}.
\end{gather}
\end{theorem}
\begin{proof}
Differentiate~\eqref{eq.y8vfuxa} with respect to $z$.
\end{proof}

\begin{lemma}
If $z\in\mathbb C\setminus\mathbb Z^{-}$, then
\begin{equation}\label{eq.vadi1jm}
\begin{split}
\sum_{n = 1}^\infty  {\frac{1}{{n\left( {n + 1} \right)\left( {n + 2} \right)\left( {n + 3} \right)\binom{{n + z}}{n}}}}  &= \frac{1}{{12}}z\left( {z + 1} \right)\left( {z + 2} \right)\left( {H_z^{(2)}  - \zeta (2)} \right)\\
&\qquad\qquad+ \frac{{z^2 }}{{12}} + \frac{{5z}}{{24}} + \frac{1}{{18}}.
\end{split}
\end{equation}
\end{lemma}
\begin{proof}
Sum~\eqref{eq.dhruknq} over $z$ using~\eqref{bs},~\eqref{eq.izgpi03} and~\cite[Equation (3.6)]{adegoke16}
\begin{equation*}
\sum_{z = 1}^r {z^2 H_z^{(2)} }  = \frac{{r\left( {r + 1} \right)\left( {2r + 1} \right)}}{6}H_r^{(2)}  - \frac{1}{6}H_r  + \frac{r}{3} - \frac{r^2}{6},
\end{equation*}
to obtain
\begin{equation}\label{eq.pv5xq4g}
\sum_{n = 1}^\infty  {\frac{1}{{n\left( {n + 3} \right)\binom{{n + z}}{n}}}}  = \frac{1}{6}z\left( {z - 1} \right)\left( {z - 2} \right)\left( {H_z^{(2)}  - \zeta (2)} \right) + \frac{1}{6}z^2  - \frac{7}{{12}}z + \frac{{11}}{{18}}.
\end{equation}
Identity~\eqref{eq.vadi1jm} is obtained by adding~\eqref{eq.pv5xq4g} and \eqref{eq.lmcv0zf} and subtracting~\eqref{eq.dhruknq}.
\end{proof}

\begin{theorem}
If $m$ is a non-negative integer, then
\begin{equation}
\begin{split}
&\sum_{n = 1}^\infty  {\frac{{2^{2n} }}{{n\left( {n + 1} \right)\left( {n + 2} \right)\left( {n + 3} \right)\binom{{2\left( {n + m} \right)}}{{n + m}}}\binom{n + m}m}}\\
&\qquad  = \frac{1}{{12\binom{{2m}}{m}}}\left( {m - \frac{1}{2}} \right)\left( {m + \frac{1}{2}} \right)\left( {m + \frac{3}{2}} \right)\left( {4O_m^{(2)}  - 3\zeta (2)} \right)\\
&\qquad\quad + \frac{1}{{4\binom{{2m}}{m}}}\left( {\frac{{m^2 }}{3} + \frac{m}{2} - \frac{1}{9}} \right).
\end{split}
\end{equation}
In particular,
\begin{gather}
\sum_{n = 1}^\infty  {\frac{{2^{2n} }}{{n\left( {n + 1} \right)\left( {n + 2} \right)\left( {n + 3} \right)\binom{{2n}}{n}}}}  = \frac{{\pi ^2 }}{{64}} - \frac{1}{{36}},\\
\sum_{n = 1}^\infty  {\frac{{2^{2n} }}{{n\left( {n + 1} \right)^2 \left( {n + 2} \right)\left( {n + 3} \right)\binom{{2\left( {n + 1} \right)}}{{n + 1}}}}}  = \frac{{29}}{{72}} - \frac{{5\pi ^2 }}{{128}},\\
\sum_{n = 1}^\infty  {\frac{{2^{2n + 1} }}{{n\left( {n + 1} \right)^2 \left( {n + 2} \right)^2 \left( {n + 3} \right)\binom{{2\left( {n + 2} \right)}}{{n + 2}}}}}  = \frac{{65}}{{72}} - \frac{{35\pi ^2 }}{{384}}.
\end{gather}
\end{theorem}
\begin{proof}
Set $z=m - 1/2$ in~\eqref{eq.vadi1jm}.
\end{proof}

\begin{theorem}
If $z\in\mathbb C\setminus\mathbb Z^{-}$, then
\begin{equation}\label{eq.qaytndb}
\begin{split}
&\sum_{n = 1}^\infty  {\frac{{H_{n + z} }}{{n\left( {n + 1} \right)\left( {n + 2} \right)\left( {n + 3} \right)\binom{{n + z}}{n}}}} \\
&\qquad=   \left( {\frac{1}{{12}}z\left( {z + 1} \right)\left( {z + 2} \right)H_z  - \frac{{z\left( {z + 2} \right)}}{4} - \frac{1}{6}} \right)\left( {H_z^{(2)}  - \zeta (2)} \right) \\
&\qquad\qquad+ \frac{1}{6}z\left( {z + 1} \right)\left( {z + 2} \right)\left( {H_z^{(3)}  - \zeta (3)} \right) + \left( {\frac{{z^2 }}{{12}} + \frac{{5z}}{{24}} + \frac{1}{{18}}} \right)H_z  - \frac{z}{6} - \frac{5}{{24}}.
\end{split}
\end{equation}
In particular,
\begin{equation}
\sum_{n = 1}^\infty  {\frac{{H_n }}{{n\left( {n + 1} \right)\left( {n + 2} \right)\left( {n + 3} \right)}}}  = \frac{{\pi ^2 }}{{36}} - \frac{5}{{24}}.
\end{equation}
\end{theorem}
\begin{proof}
Differentiate~\eqref{eq.vadi1jm} to obtain
\begin{equation}\label{eq.z5mqka6}
\begin{split}
\sum_{n = 1}^\infty  {\frac{{H_{n + z}  - H_z }}{{n\left( {n + 1} \right)\left( {n + 2} \right)\left( {n + 3} \right)\binom{{n + z}}{n}}}}&= - \left( {\frac{{z\left( {z + 2} \right)}}{4} + \frac{1}{6}} \right)\left( {H_z^{(2)}  - \zeta (2)} \right)\\
&\qquad + \frac{1}{6}z\left( {z + 1} \right)\left( {z + 2} \right)\left( {H_z^{(3)}  - \zeta (3)} \right) - \frac{z}{6} - \frac{5}{{24}},
\end{split}
\end{equation}
and hence~\eqref{eq.qaytndb} by a repeated use of~\eqref{eq.vadi1jm}.
\end{proof}

\begin{theorem}
If $m$ is a non-negative integer, then
\begin{equation}
\begin{split}
&\sum_{n = 1}^\infty  {\frac{{2^{2n} O_{n + m} }}{{n\left( {n + 1} \right)\left( {n + 2} \right)\left( {n + 3} \right)\binom{{2\left( {n + m} \right)}}{{n + m}}\binom{{n + m}}{m}}}} \\
&\qquad = \frac{1}{{2\binom{{2m}}{m}}}\left( {\frac{1}{6}\left( {m - \frac{1}{2}} \right)\left( {m + \frac{1}{2}} \right)\left( {m + \frac{3}{2}} \right)O_m  - \frac{{m^2 }}{4} - \frac{m}{4} + \frac{1}{{48}}} \right)\left( {4O_m^{(2)}  - 3\zeta (2)} \right)\\
&\qquad\quad + \frac{1}{{12\binom{{2m}}{m}}}\left( {m - \frac{1}{2}} \right)\left( {m + \frac{1}{2}} \right)\left( {m + \frac{3}{2}} \right)\left( {8O_m^{(3)}  - 7\zeta (3)} \right)\\
&\qquad\quad + \frac{1}{{2\binom{{2m}}{m}}}\left( {\frac{1}{{36}}\left( {6m^2  + 9m - 2} \right)O_m  - \frac{m}{6} - \frac{1}{8}} \right).
\end{split}
\end{equation}
In particular,
\begin{gather}
\sum_{n = 1}^\infty  {\frac{{2^{2n} O_n }}{{n\left( {n + 1} \right)\left( {n + 2} \right)\left( {n + 3} \right)\binom{{2n}}{n}}}}  = \frac{7}{{32}}\,\zeta (3) - \frac{{\pi ^2 }}{{192}} - \frac{1}{{16}},\\
\sum_{n = 1}^\infty  {\frac{{2^{2n} O_{n + 1} }}{{n\left( {n + 1} \right)^2 \left( {n + 2} \right)\left( {n + 3} \right)\binom{{2\left( {n + 1} \right)}}{{n + 1}}}}}  = \frac{{137}}{{288}} + \frac{{\pi ^2 }}{{48}} - \frac{{35}}{{64}}\zeta (3).
\end{gather}
\end{theorem}
\begin{proof}
Write $m-1/2$ for $z$ in~\eqref{eq.z5mqka6} and use Lemmata~\ref{lem.ho} and~\ref{lem.bin}.
\end{proof}

\begin{theorem}
If $z\in\mathbb C\setminus\mathbb Z^{-}$, then
\begin{equation}
\begin{split}
\sum_{n = 1}^\infty  {\frac{{\left( {H_{n + z}  - H_z } \right)^2  + H_{n + z}^{(2)}  - H_z^{(2)} }}{{n\left( {n + 1} \right)\left( {n + 2} \right)\left( {n + 3} \right)\binom{{n + z}}{z}}}}  &= \frac{1}{2}\left( {z + 1} \right)\left( {H_z^{(2)}  - \zeta (2)} \right)\\
&\qquad- \left( {z\left( {z + 2} \right) + \frac{2}{3}} \right)\left( {H_z^{(3)}  - \zeta (3)} \right)\\
&\qquad\quad+ \frac{1}{2}z\left( {z + 1} \right)\left( {z + 2} \right)\left( {H_z^{(4)}  - \zeta (4)} \right) + \frac{1}{6}.
\end{split}
\end{equation}
In particular,
\begin{gather}
\sum_{n = 1}^\infty  {\frac{{H_n^2  + H_n^{(2)} }}{{n\left( {n + 1} \right)\left( {n + 2} \right)\left( {n + 3} \right)}}}  = \frac{2}{3}\,\zeta (3) - \frac{{\pi ^2 }}{{12}} + \frac{1}{6},\\
\sum_{n = 1}^\infty  {\frac{{2^{2n} \left( {O_n^2  + O_n^{(2)} } \right)}}{{n\left( {n + 1} \right)\left( {n + 2} \right)\left( {n + 3} \right)\binom{{2n}}{n}}}}  = \frac{{\pi ^4 }}{{128}} - \frac{{\pi ^2 }}{{32}} - \frac{7}{{48}}\,\zeta (3) + \frac{1}{{24}}.
\end{gather}
\end{theorem}
\begin{proof}
Differentiate~\eqref{eq.z5mqka6} with respect to $z$.
\end{proof}

% \section{Acknowledgments}

% The authors thank the referee and the Editor-in-Chief for their very helpful comments that improved the exposition of the paper.


\begin{thebibliography}{99}

\bibitem{Abel}
U. Abel, A short proof of the binomial identities of Frisch and Klamkin, \emph{J. Integer Seq.} {\bf 23} (2020), Article 20.7.1.

\bibitem{Adegoke}
K. Adegoke and R. Frontczak, Some notes on an identity of Frisch, preprint, 2024, submitted. https://arxiv.org/pdf/2405.10978

\bibitem{Adegoke2}
K. Adegoke and R. Frontczak, Series associated with a forgotten identity of N\"orlund, 2024, submitted. https://arxiv.org/pdf/2410.22343

\bibitem{adegoke16} 
K. Adegoke and O. Layeni, New finite and infinite summation identities involving the generalized harmonic numbers, 
\emph{J. Analysis Number Theory} {\bf 4}:1 (2016), 49--60.

\bibitem{bowen95}
D. Borwein and J. M. Borwein, On an intriguing integral and some series related to $\zeta(4)$,
\emph{Proc. Amer. Math. Soc.} {\bf 193} (1995), 1191--1198.

\bibitem{Boyadzhiev1}
K. N. Boyadzhiev, Power series with binomial sums and asymptotic expansions, \emph{Int. Journal Math. Analysis} {\bf 8}:28 (2014),
1389--1414.

\bibitem{Boyadzhiev2}
K. N. Boyadzhiev, New series identities with Cauchy, Stirling, and harmonic numbers, and Laguerre polynomials,
\emph{J. Integer Seq.} {\bf 23} (2020), Article 20.11.7.

\bibitem{Boyadzhiev3}
K. N. Boyadzhiev, Stirling numbers and inverse factorial series, \emph{Contrib. Math.} {\bf 7} (2023), 24-33.

\bibitem{chu12} 
W. Chu, Infinite series identities on harmonic numbers, \emph{Results Math.} {\bf 61} (2012), 209--221.

\bibitem{guillera08}
J. Guillera and J. Sondow, Double integrals and infinite products for some classical constants via analytic continuations of Lerch's transcendent , \emph{Ramanujan J.} {\bf 16} (2008), 247--270.

\bibitem{nimbran19} 
S. S. Nimbran and A. Sofo, New interesting Euler sums, \emph{J. Class. Anal.} {\bf 15} (2019), 9--22.

\bibitem{rocket81}
A. M. Rockett, Sums of the inverses of binomial coefficients, \emph{Fibonacci Quart.} {\bf 19}:5 (1981), 433--437.

\bibitem{Sofo1}
A. Sofo, General properties involving reciprocals of binomial coefficients, \emph{J. Integer Seq.} {\bf 9} (2006), Article 06.4.5.

\bibitem{Sofo2}
A. Sofo, Integrals and polygamma representations for binomial sums, \emph{ J. Integer Seq.} {\bf 13} (2010), Article 10.2.8.

\bibitem{sofo15} 
A. Sofo and H. M. Srivastava, A family of shifted harmonic sums, \emph{Ramanujan J.} {\bf 37} (2015), 89--108.

\bibitem{Srivastava}
H. M. Srivastava and J. Choi, \emph{Series Associated with the Zeta and Related Functions}, Springer Science+Media, B.V., 2001.

\bibitem{xu17} 
C. Xu, Identities for the shifted harmonic numbers and binomial coefficients, \emph{Filomat} {\bf 31}:19 (2017), 6071--6086.

\end{thebibliography}
\end{document}